\documentclass[a4paper,12pt]{article}

\usepackage{amsfonts}
\usepackage{amsmath}
\usepackage{amsthm}
\usepackage{hyperref}

\newtheorem{theorem}{Theorem}
\newtheorem{corollary}[theorem]{Corollary}
\newtheorem{lemma}[theorem]{Lemma}
\newtheorem{example}{Example}

\DeclareMathOperator{\kcon}{kcon}
\DeclareMathOperator{\gkcon}{kcon}
\DeclareMathOperator{\C}{C}
\DeclareMathOperator{\GC}{GC}

\begin{document}
\title{Enumerating k-connected blocks and gk-connected blocks in words}

\author{Walaa Asakly and Noor Kezil\\
Department of Mathematics, Braude college, Karmiel, Israel\\
Department of Mathematics, University of Haifa , Haifa, Israel\\
{\tt walaaa@braude.ac.il}\\
{\tt nkizil02@campus.haifa.ac.il}}
\date{\small }
\maketitle

\begin{abstract}
 We define two new statistics on words: the k-connector and the gk-connector. For a word $\pi = \pi_1\pi_2\cdots\pi_n$ of length $n$ over the alphabet $[k]$, a k-connector is defined as an ordered pair $(\pi_j, \pi_{j+1})$ where $1 \leq j \leq n-1$ and $\pi_j + \pi_{j+1} = k$. Conversely, a gk-connector is defined as an ordered pair $(\pi_j, \pi_{j+1})$ where $1 \leq j \leq n-1$ and $\pi_j + \pi_{j+1} > k$. We investigate the enumeration of partitions based on these statistics, providing generating functions and explicit formulas. 

\medskip

\noindent{\bf Keywords}:  k-connector, gk-connector, Cramer's method and generating function.
\end{abstract}

\section{Introduction}
 Let $[k]=\{1,2,\ldots,k\}$ be an alphabet over $k$
letters. A {\em word} $\pi$ of length $n$ over alphabet $[k]$
is an element of $[k]^n$ and is also called {\em word of length
$n$ on $k$ letters} or {\em $k$-ary word of length $n$}. The
number of the words of length $n$ over alphabet $[k]$ is $k^n$.
Statistics in words attract researchers, particularly, the study of 
avoidance patterns in words, see \cite{HM}. For example, Burstein and Mansour \cite{AT} enumerated 
the number of words which contain $111$, $222$, and $333$ as substrings a total
 of $r$ times. Kitaev, Mansour
and Remmel \cite{STJ} found the number of subword patterns $12$,
, $11$ and $21$ in words that have a prescribe first
element, in words. In addition, Mansour \cite{TM}
enumerated the number of subword patterns $121$, $132$ or $231$ and subword patterns $212$, $213$ or $312$.
in words of length $n$ over alphabet $k$. 
Knopfmacher, Munagi and Wagner \cite{AAS} found the mean and variance of the
$k$-ary words of length $n$ according to the number of {\em $p$-successions}, ($p$-succession in
 a $k$-ary word
$\omega_1\omega_2\cdots\omega_n$ of length $n$ is two consecutive letters
of the form $(\omega_i, \omega_i+ p)$, where $i = 1, 2, \cdots, n-1$). 
Mansour and Munagi \cite{TA} found the generating function for $k$-ary words of length $n$ according to the
statistic $con$ ($con(\pi)$ is defined to be the number of elements $\pi_i$ such that $\pi_i=\pi_{i-1}+1$ with $i=2,3\ldots ,n$).
In this paper,  We define two new statistics on words: the k-connector and the gk-connector. For a word $\pi = \pi_1\pi_2\cdots\pi_n$ of length $n$ over the alphabet $[k]$, a k-connector at $i$ is defined as an ordered pair $(\pi_i, \pi_{i+1})$ where $1 \leq i \leq n-1$ and $\pi_i + \pi_{i+1} = k$. A gk-connector at $i$ is defined as an ordered pair $(\pi_i, \pi_{i+1})$ where $1 \leq i \leq n-1$ and $\pi_i + \pi_{i+1} > k$.  For example, the word $\pi=143114$ of length $6$ over the alphabet $[4]$ has one k-connector (at $i=3$) and three gk-connectors ( at $i=1$, $i=2$ and $i=5$).
Let $\kcon(\pi)$ denote the number of k-connectors, and let $\gkcon(\pi)$ denote the number of gk-connectors.

\section{k-connectors}
\subsubsection{The generating function for the number of words of length $n$ over the alphabet $[k]$ according to the number of k-connectors}
Let $\kcon(\pi)$ denote the number of k-connectors.
Consider  the generating function for the number of words of $n$
according to the statistics  $\kcon$ to be $\C_k(x,q)$. That is,
$$\C_k(x,q)=\sum_{n\geq 0}\sum_{\pi \in [k]^n}x^n q^{\kcon(\pi)}.$$

\begin{theorem}\label{Th1}
 The generating function for the number of words of length $n$ over the alphabet $[k]$,
according to the number of k-connectors, is given by
\begin{align*}
\C_k(x,q)=\frac{1}{1-x-\left( \frac{x+x^2(q-1)}{1-x^2(q-1)^2}\right)(k-1)}.
\end{align*}
\end{theorem}
\begin{proof}
Let $C_k(x,q|i_1i_2\ldots i_s)$ denote the generating function for the number of $k$-ary words $\pi=\pi_1\pi_2\ldots\pi_n$ of length $n$ according to the statistic $\kcon$, where $\pi_{n-s+1}\pi_{n-s+2}\ldots\pi_n=i_1i_2\ldots i_s$. From the definitions we have
\begin{equation}\label{eq1}
C_k(x,q)=1+\sum_{i=1}^{k}C_k(x,q|i)=1+\sum_{i=1}^{k-1}C_k(x,q|i)+C_k(x,q|k).
\end{equation}
For $1\leq i\leq k-1$
\begin{align*}
C_k(x,q|i)&=x+\sum_{j=1, j\neq k-i}^{k}C_k(x,q|ji)+qxC_k(x,q|k-i)\\
&=x+x\sum_{j=1}^{k-1}C_k(x,q|j)-xC_k(x,q|k-i)+qxC_k(x,q|k-i)+xC_k(x,q|k).
\end{align*}
By (\ref{eq1}) we get
\begin{align*}
C_k(x,q|i)&=x+x\left(C_k(x,q)-1-C_k(x,q|k)\right)+x(q-1)C_k(x,q|k-i)+xC_k(x,q|k)\\
&=xC_k(x,q)+x(q-1)C_k(x,q|k-i).
\end{align*}
Therfore, by induction on $i$, we have
\begin{align*}
C_k(x,q|i)=xC_k(x,q)+x(q-1)\left(xC_k(x,q)+x(q-1)C_k(x,q|i)\right).
\end{align*}
Then
\begin{align*}
C_k(x,q|i)=C_k(x,q)\left(\frac{x+x^2(q-1)}{1-x^2(q-1)^2}\right).
\end{align*}
By summing over $1\leq i\leq k-1$, and by using (\ref{eq1}) we get
\begin{align*}
C_k(x,q)-1-C_k(x,q|k)=C_k(x,q)(k-1)\left(\frac{x+x^2(q-1)}{1-x^2(q-1)^2}\right).
\end{align*}
By substituting $C_k(x,q|k)=xC_k(x,q)$, we obtain the required result.
\end{proof}

Let $\kcon(n)$ denote the number of k-connected blocks in $C_n$.
\begin{lemma}\label{lem1}
The generation functions for the sequence $\kcon(n)$ is given by
 \begin{equation}\label{eq4}
 \frac{\partial }{\partial q}C_k(x,q)\mid_{q=1}=\frac{(k-1)x^2}{\left(1-x-(k-1)x\right)^2}
 \end{equation}
\end{lemma}
Using a computer algebra system such as Maple, we determined the coefficients of $x^n$ in (\ref{eq4}) and obtained the following result.
\begin{corollary}\label{cor1}
The number of the k-connectors blocks over all words of length $n$ over the alphabet $[k]$ is given by
\begin{equation*}
(k-1)(n-1)k^{n-2}.
\end{equation*}
\end{corollary}
For instance, when $n=3$ and $k=2$, we have three words that have $2$-connectors: $111$ has two $2$-connectors; $112$ has one $2$-connectors; and $211$ has one $2$-connectors. The number of $2$-connectors is $4$.

\section{gk-connectors}
\subsubsection{The generating function for the number of words of length $n$ over the alphabet $[k]$ according to the number of gk-connectors}
Let $\gkcon(\pi)$ denote the number of gk-connectors.
Consider  the generating function for the number of words of $n$
according to the statistics  $\gkcon$ to be $\GC_k(x,q)$. That is,
$$\GC_k(x,q)=\sum_{n\geq 0}\sum_{\pi \in [k]^n}x^n q^{\gkcon(\pi)}.$$
\begin{theorem}\label{Th2}
 The generating function for the number of words of length $n$ over the alphabet $[k]$,
according to the number of gk-connectors, satisfies
\begin{align*}
GC_k(x,q)=\frac{1+x(q-1)\sum_{i=1}^{k}\sum_{j=k-i+1}^{k}GC(x,q|j)}{1-kx}.
\end{align*}
\end{theorem}
\begin{proof}
Let $GC_k(x,q|i_1i_2\ldots i_s)$ denote the generating function for the number of $k$-ary words $\pi=\pi_1\pi_2\ldots\pi_n$ of length $n$ according to the statistic $\gkcon$, where $\pi_{n-s+1}\pi_{n-s+2}\ldots\pi_n=i_1i_2\ldots i_s$. From the definitions we have
\begin{equation}\label{eq2}
GC_k(x,q)=1+\sum_{i=1}^{k}GC_k(x,q|i).
\end{equation}
For $1\leq i\leq k$
\begin{align*}
GC_k(x,q|i)&=x+\sum_{j=1, j\neq k}^{k}GC_k(x,q|ji)\\
&=x+\sum_{j=1}^{k-i}GC_k(x,q|ji)+\sum_{j=k-i+1}^{k}GC_k(x,q|ji)\\
&=x+x\sum_{j=1}^{k-i}GC_k(x,q|j)+qx\sum_{j=k-i+1}^{k}GC_k(x,q|j).
\end{align*}
By (\ref{eq2}) we get
\begin{align*}
GC_k(x,q|i)&=x+x\left(\sum_{j=1}^{k}GC_k(x,q|j)-\sum_{j=k-i+1}^{k}GC_k(x,q|j)\right)+qx\sum_{j=k-i+1}^{k}GC_k(x,q|j)\\
&=x+x\left(GC_k(x,q)-1-\sum_{j=k-i+1}^{k}GC_k(x,q|j)\right)+qx\sum_{j=k-i+1}^{k}GC_k(x,q|j).
\end{align*}
Which leads to
\begin{equation}\label{eq3}
GC_k(x,q|i)=xGC_k(x,q)+x(q-1)\sum_{j=k-i+1}^{k}GC_k(x,q|j).
\end{equation}
By summing over $1\leq i\leq k$, and by using (\ref{eq2}) we get
\begin{align*}
GC_k(x,q)=\frac{1+x(q-1)\sum_{i=1}^{k}\sum_{j=k-i+1}^{k}G(x,q|j)}{1-kx}.
\end{align*}
\end{proof}
\begin{theorem}\label{Th3}
 The generating function for the number of words of length $n$ over the alphabet $[k]$,
according to the number of gk-connectors, where $k$ is even number is given by,
\begin{align*}
GC_k(x,q)=\frac{\sum_{i=0}^{k}{(-1)^{\lfloor\frac{i+1}{2}\rfloor}}\binom{\lfloor\frac{k-i}{2}\rfloor+i}{i}b^i}
{\sum_{i=0}^{k}{(-1)^{\lfloor\frac{i+1}{2}\rfloor}}\binom{\lfloor\frac{k-i}{2}\rfloor+i}{i}b^i-x\sum_{j=0}^{k-1}\left( 1+{(-1)^{\lfloor\frac{j}{2}\rfloor}}\sum_{i=1}^{j}(-1)^{\lfloor\frac{j-i}{2}\rfloor}\binom{\lfloor\frac{j-i}{2}\rfloor+i}{i}b^i\right)}
\end{align*}
Where $b=x(q-1)$.
\end{theorem}

\begin{proof}

 By summing over $1\leq i\leq k$ for $k$ even, and by using (\ref{eq3}) we get,
\begin{align*}
\hspace{-0.1cm}
\left\{\begin{array}{ll}
GC_k(x,q|1)-bGC_k(x,q|k)&=a\\
GC_k(x,q|2)-bGC_k(x,q|k-1)-bGC_k(x,q|k)&=a\\
&\vdots\\
-bGC_k(x,q|\lceil{\frac{k+1}{2}}\rceil-1)+(1-b)GC_k(x,q|\lceil{\frac{k+1}{2}}\rceil)-bGC_k(x,q|\lceil{\frac{k+1}{2}}\rceil+1)\dots-bGC_k(x,q|k)&=a\\
&\vdots\\
-bGC_k(x,q|\lceil{\frac{k+1}{2}}\rceil-2)-bGC_k(x,q|\lceil{\frac{k+1}{2}}\rceil-1)-bGC_k(x,q|\lceil{\frac{k+1}{2}}\rceil)
\\+(1-b)GC_k(x,q|\lceil{\frac{k+1}{2}}\rceil+1)-bGC_k(x,q|\lceil{\frac{k+1}{2}}\rceil+2)\dots-bGC_k(x,q|k)&=a\\
&\vdots\\
-bGC_k(x,q|1)-bGC_k(x,q|2)\dots-bGC_k(x,q|k-1)
+(1-b)GC_k(x,q|k)&=a\\
\end{array}\right. ,
\end{align*}
where $a=xGC_k(x,q)$ and $b=x(q-1)$.
 Note that we can write the above system of equations as a linear system of $k$ equations in $k$ variables. Therefore, it
can be written in a matrix form as follows
$$A \left(\begin{array}{l}
 GC_k(x,q|1)\\
GC_k(x,q|2)\\
 \vdots \\
GC_k(x,q|k) \end{array}\right)=
 \left(\begin{array}{c}
 a\\
 a\\
 \vdots \\
 a \end{array}\right),$$
where
\begin{equation}
A= \left( \begin{array}{llllllll}
 1& 0& 0& \cdots & \cdots& \cdots &0 & -b \\[1ex]
0& 1& 0& \cdots & \cdots &0 &-b &-b  \\[1ex]
 &  \ddots &  \ddots & \nonumber\\[1ex]
0& \cdots & 0 &1 & -b & \cdots &\cdots &-b\\ [1ex]
0& \cdots & -b &1-b & -b & \cdots &\cdots &-b \hspace{1em} \text{(Row${\lceil\frac{k+1}{2}}\rceil$)}\\[1ex] 
 0& \cdots & -b &-b & 1-b & -b &\cdots &-b \\[1ex] 
 &  \ddots &  \ddots & \\[1ex]
-b & -b & \cdots& \cdots & \cdots& \cdots& -b  & 1-b 
\end{array} \right)
\end{equation}

\begin{example}
For $k=2$
\[
A= \left( \begin{array}{ll}
 1 & -b \\
 -b & 1 - b 
\end{array} \right)
.\]
For $k=4$
\[
A= \left( \begin{array}{llll}
 1 & 0& 0&-b\\
 0 & 1& -b&-b\\
0 & -b& 1-b&-b\\
 -b & - b&-b&1-b 
\end{array} \right)
.\]
For $k=6$
\[
A= \left( \begin{array}{llllll}
 1 & 0& 0&0&0&-b\\
 0 & 1&0&0& -b&-b\\
0 & 0&1&-b& -b&-b\\
0 & 0&-b&1-b& -b&-b\\
0 & -b&-b&-b& 1-b&-b\\
-b & -b&-b&-b& -b&1-b
\end{array} \right)
.\]
\end{example}
By using Cramer's method we obtain that the generating function
$G(x,q|\ell)$ is given by
$$\frac{\det A_{\ell}}{\det A}$$

Where, $A_\ell$ is obtained by replacing column number $\ell$ in matrix $A$ with $\left(\begin{array}{c}
 a\\
 a\\
 \vdots \\
 a \end{array}\right)$.  
 \begin{lemma}\label{L1}
For all $k$ even  
$$\det A=\sum_{i=0}^{k}{(-1)^{\lfloor\frac{i+1}{2}\rfloor}}\binom{\lfloor\frac{k-i}{2}\rfloor+i}{i}b^i$$
\end{lemma}
\begin{proof}
For $k=2$
\[
\det A= \left| \begin{array}{ll}
 1 & -b \\
 -b & 1 - b 
\end{array} \right|=-b^2-b+1
.\]
For $k=4$
\[
\det A= \left|\begin{array}{llll}
 1 & 0& 0&-b\\
 0 & 1& -b&-b\\
0 & -b& 1-b&-b\\
 -b & - b&-b&1-b 
\end{array} \right|=b^4+b^3-3b^2-2b+1
.\]
For $k=6$
\[
\det A= \left| \begin{array}{llllll}
 1 & 0& 0&0&0&-b\\
 0 & 1&0&0& -b&-b\\
0 & 0&1&-b& -b&-b\\
0 & 0&-b&1-b& -b&-b\\
0 & -b&-b&-b& 1-b&-b\\
-b & -b&-b&-b& -b&1-b
\end{array} \right|=-b^6 - b^5 + 5b^4 + 4b^3 - 6b^2 - 3b + 1
.\]
  By guessing the formula and using induction over $k$ even we get the result.
\end{proof}

\begin{lemma}\label{L2}
For  $1\leq\ell\leq k$, where $k$ is even  
$$\det A_{\ell}=a\sum_{j=0}^{\ell-1}\left( 1+{(-1)^{\lfloor\frac{j}{2}\rfloor}}\sum_{i=1}^{j}(-1)^{\lfloor\frac{j-i}{2}\rfloor}\binom{\lfloor\frac{j-i}{2}\rfloor+i}{i}b^i\right)$$
\end{lemma}
\begin{proof}
For $k=2$
\[
\det A_1= \left| \begin{array}{ll}
 a & -b \\
 a & 1 - b 
\end{array} \right|=a
.\]

\[
\det A_2= \left| \begin{array}{ll}
 1 & a \\
 -b & a 
\end{array} \right|=a(-b+1)
.\]
For $k=4$
\[
\det A_1= \left|\begin{array}{llll}
a& 0&0&-b\\
a&1&-b&-b\\
a&-b&1-b&-b\\
a&-b&-b&1-b 
\end{array} \right|=a(-b^2-b+1)
.\]

\[
\det A_2= \left|\begin{array}{llll}
1&a&0&-b\\
0&a&-b&-b\\
0&a&1-b&-b\\
-b&a&-b&1-b 
\end{array} \right|=a
.\]

\[
\det A_3= \left|\begin{array}{llll}
1&0&a&-b\\
0&1&a&-b\\
0&-b&a&-b\\
-b&-b&a&1-b 
\end{array} \right|=a(b+1)
.\]

\[
\det A_4= \left|\begin{array}{llll}
 1 & 0& 0&a\\
 0 & 1& -b&a\\
0 & -b& 1-b&a\\
 -b & - b&-b&a
\end{array} \right|=a(-b^3-b^2+2b+1)
.\]
For $k=6$
\[
\det A_1= \left| \begin{array}{llllll}
 a & 0& 0&0&0&-b\\
 a & 1&0&0& -b&-b\\
a& 0&1&-b& -b&-b\\
a & 0&-b&1-b& -b&-b\\
a & -b&-b&-b& 1-b&-b\\
a & -b&-b&-b& -b&1-b
\end{array} \right|=a(b^4+b^3-3b^2-2b+1)
.\]
\[
\det A_2= \left| \begin{array}{llllll}
 1 & a& 0&0&0&-b\\
 0 & a&0&0& -b&-b\\
0 & a&1&-b& -b&-b\\
0 & a&-b&1-b& -b&-b\\
0 & a&-b&-b& 1-b&-b\\
-b & a&-b&-b& -b&1-b
\end{array} \right|=a(-b^2-b+1)
.\]
\[
\det A_3= \left| \begin{array}{llllll}
 1 & 0& a&0&0&-b\\
 0 & 1&a&0& -b&-b\\
0 & 0&a&-b& -b&-b\\
0 & 0&a&1-b& -b&-b\\
0 & -b&a&-b& 1-b&-b\\
-b & -b&a&-b& -b&1-b
\end{array} \right|=a
.\]
\[
\det A_4= \left| \begin{array}{llllll}
 1 & 0& 0&a&0&-b\\
 0 & 1&0&a& -b&-b\\
0 & 0&1&a& -b&-b\\
0 & 0&-b&a& -b&-b\\
0 & -b&-b&a& 1-b&-b\\
-b & -b&-b&a& -b&1-b
\end{array} \right|=a(b+1)
.\]
\[
\det A_5= \left| \begin{array}{llllll}
 1 & 0& 0&0&a&-b\\
 0 & 1&0&0&a&-b\\
0 & 0&1&-b&a&-b\\
0 & 0&-b&1-b&a&-b\\
0 & -b&-b&-b&a&-b\\
-b & -b&-b&-b&a&1-b
\end{array} \right|=a(-b^3-b^2+2b+1)
.\]
\[
\det A_6= \left| \begin{array}{llllll}
 1 & 0& 0&0&0&a\\
 0 & 1&0&0& -b&a\\
0 & 0&1&-b& -b&a\\
0 & 0&-b&1-b& -b&a\\
0 & -b&-b&-b& 1-b&a\\
-b & -b&-b&-b& -b&a
\end{array} \right|=a(b^5+b^4-4b^3-3b^2+3b+1)
.\]
By guessing the formula and using induction over $k$ even we get the result.
\end{proof}
Using (\ref{eq2}), we derive that the generating function $GC_k(x,q)$ is given by
\begin{align*}
GC_k(x,q)=1+\frac{1}{\det A}\sum_{\ell=1}^{k}{\det A_{\ell}},
\end{align*}
for $k$ even.

By using Lemma  \ref{L2} in the above equation we obtain,
\begin{align*}
GC_k(x,q)=\frac{\det A}{\det A-x\sum_{j=0}^{k-1}\left( 1+{(-1)^{\lfloor\frac{j}{2}\rfloor}}\sum_{i=1}^{j}(-1)^{\lfloor\frac{j-i}{2}\rfloor}\binom{\lfloor\frac{j-i}{2}\rfloor+i}{i}b^i\right)},
\end{align*}
by applying Lemma \ref{L1} we get the required result.
\end{proof}
\begin{corollary}
The number of the gk-connectors blocks over all words of length $n$ over the alphabet $[k]$  where $k$ is even number, is given by
$$\frac{(k+1)}{2}(n-1)k^{n-1}.$$
\end{corollary}
\begin{proof}
To find the total number of the gk-connectors blocks over all words of length $n$ over the alphabet $[k]$, we differentiate the generating 
function $GC_k(x,q)$ according to $q$, substitute $q=1$, and extract the coefficients of $x^n$.
Let
 $$f(x,q)=\sum_{i=0}^{k}{(-1)^{\lfloor\frac{i+1}{2}\rfloor}}\binom{\lfloor\frac{k-i}{2}\rfloor+i}{i}x^i(q-1)^i,$$ and $$g(x,q)=f(x,q)-x\sum_{j=0}^{k-1}\left( 1+{(-1)^{\lfloor\frac{j}{2}\rfloor}}\sum_{i=1}^{j}(-1)^{\lfloor\frac{j-i}{2}\rfloor}\binom{\lfloor\frac{j-i}{2}\rfloor+i}{i}x^i(q-1)^i\right).$$
Note that, according to Theorem \ref{Th3}, we have
 $$GC_k(x,q)=\frac{f(x,q)}{f(x,q)-g(x,q)}.$$ 
Differentiating  $GC_k(x,q)$ according to $q$ and substituting $q=1$, we obtain,
\begin{equation}\label{eq 1}
\frac{\partial GC_k(x,q)}{\partial q}|_{q=1}=\frac{\frac{{\partial f(x,q)}}{{\partial q}}|_{q=1}(f(x,1)-g(x,1))-\frac{\partial( {f(x,q)-g(x,q)})}{\partial q}|_{q=1}f(x,1)}{(f(x,1)-g(x,1))^2}.
\end{equation}
Where,
$f(x,1)-g(x,1)=1-kx$, $\frac{{\partial f(x,q)}}{{\partial q}}|_{q=1}=-\lfloor \frac{k+1}{2}\rfloor x$, and 
$$\frac{\partial {g(x,q)}}{\partial q}|_{q=1}=x^2\sum_{j=1}^{k-1}(-1)^{\lfloor \frac{j}{2} \rfloor + \lfloor \frac{j-1}{2} \rfloor}(\lfloor \frac{j-1}{2} \rfloor +1).$$
Using the facts
$ \lfloor \frac{j}{2} \rfloor + \lfloor \frac{j-1}{2} \rfloor=j-1$
 and $\lfloor \frac{j-1}{2} \rfloor +1=\lfloor \frac{j+1}{2} \rfloor$
leads to $$\frac{\partial {g(x,q)}}{\partial q}|_{q=1}=x^2\sum_{j=1}^{k-1}(-1)^{j-1}\lfloor \frac{j+1}{2} \rfloor. $$
Using the identity $\sum_{j=1}^{k-1}(-1)^{j-1}\lfloor \frac{j+1}{2} \rfloor=\lfloor  \frac{k}{2}\rfloor$ 
in (\ref{eq 1}), leads to
\begin{align*}
\frac{\partial GC_k(x,q)}{\partial q}|_{q=1}&= \left( \lfloor \frac{k+1}{2} \rfloor k+\lfloor \frac{k}{2} \rfloor \right)\frac{x^2}{(1-kx)^2}\\
&=\frac{k(k+1)}{2} \frac{x^2}{(1-kx)^2}\\
&=\frac{k(k+1)}{2}\sum_{i=0}^{\infty}\binom{1+n}{n}(kx)^n.
\end{align*}
By finding the coefficients of $x^{n-2}$ we get the desired result.
\end{proof}

\begin{theorem}\label{Th4}
 The generating function for the number of words of length $n$ over the alphabet $[k]$,
according to the number of gk-connectors, where $k$ is odd number is given by,
\begin{align*}
GC_k(x,q)=\frac{\sum_{i=0}^{k}{(-1)^{\lfloor\frac{i+1}{2}\rfloor}}\binom{\lfloor\frac{k-i}{2}\rfloor+i}{i}b^i}{\sum_{i=0}^{k}{(-1)^{\lfloor\frac{i+1}{2}\rfloor}}\binom{\lfloor\frac{k-i}{2}\rfloor+i}{i}b^i-x\sum_{j=0}^{k-1}\left(1+{(-1)^{\lfloor\frac{j+1}{2}\rfloor}}\sum_{i=1}^{j}(-1)^{\lfloor\frac{j-i+1}{2}\rfloor}
\binom{\lfloor\frac{j-i}{2}\rfloor+i}{i}b^i\right)}.
\end{align*}
Where $b=x(q-1)$.
\end{theorem}
\begin{proof}
 By summing over $1\leq i\leq k$ for $k$ odd, and by using (\ref{eq3}) we get,
\begin{align*}
\hspace{-0.1cm}
\left\{\begin{array}{ll}
GC_k(x,q|1)-bGC_k(x,q|k)&=a\\
GC_k(x,q|2)-bGC_k(x,q|k-1)-bGC_k(x,q|k)&=a\\
&\vdots\\
GC_k(x,q|\lceil{\frac{k+1}{2}}\rceil-1)-bGC_k(x,q|\lceil{\frac{k+1}{2}}\rceil+1)\dots-bGC_k(x,q|k)&=a\\
(1-b)GC_k(x,q|\lceil{\frac{k+1}{2}}\rceil)-bGC_k(x,q|\lceil{\frac{k+1}{2}}\rceil+1)\dots-bGC_k(x,q|k)&=a\\
&\vdots\\
-bGC_k(x,q|\lceil{\frac{k+1}{2}}\rceil-1)-bGC_k(x,q|\lceil{\frac{k+1}{2}}\rceil)
+(1-b)GC_k(x,q|\lceil{\frac{k+1}{2}}\rceil+1)-bGC_k(x,q|\lceil{\frac{k+1}{2}}\rceil+2)\dots-bGC_k(x,q|k)&=a\\
&\vdots\\
-bGC_k(x,q|1)-bGC_k(x,q|2)\dots-bGC_k(x,q|k-1)
+(1-b)GC_k(x,q|k)&=a\\
,
\end{array}\right.
\end{align*}
 where  $a=xGC_k(x,q)$ and $b=x(q-1)$.
 Note that we can write the above system of equations as a linear system of $k$ equations in $k$ variables. Therefore, it
can be written in a matrix form as follows
$$A \left(\begin{array}{l}
 GC_k(x,q|1)\\
GC_k(x,q|2)\\
 \vdots \\
GC_k(x,q|k) \end{array}\right)=
 \left(\begin{array}{c}
 a\\
 a\\
 \vdots \\
 a \end{array}\right),$$
where
\begin{align}
A= \left( \begin{array}{lllllllll}
 1& 0& 0& \cdots & \cdots& \cdots &\cdots &0 & -b \\[1ex]
0& 1& 0& \cdots &\cdots& \cdots &0 &-b &-b \nonumber \\[1ex]
 &  \ddots &  \ddots & \nonumber\\[1ex]
0& \cdots & 0 &1 & 0 & -b &\cdots&\cdots &-b\nonumber\\[1ex]
0& \cdots & 0 &1-b & -b & \cdots &\cdots&\cdots &-b \hspace{1em} \text{(Row${\lceil\frac{k+1}{2}}\rceil$)}\\[1ex] 
 0& \cdots &0& -b &-b & 1-b & -b &\cdots &-b \nonumber\\[1ex] 
 &  \ddots &  \ddots & \nonumber\\[1ex]
 &  \ddots &  \ddots & \nonumber\\[1ex]
-b & -b & \cdots& \cdots & \cdots&\cdots& \cdots& -b  & 1-b 
\end{array} \right)
\end{align}
\begin{example}
For $k=1$
$$A= 1.$$
For $k=3$
\[
A= \left( \begin{array}{lll}
 1 & 0&-b\\
 0 & 1-b&-b\\
-b&-b&1-b
\end{array} \right)
.\]
For $k=5$
\[
A= \left( \begin{array}{lllll}
 1 & 0& 0&0&-b\\
 0 & 1&0& -b&-b\\
0 & 0&1-b&-b& -b\\
0 & -b&-b&1-b& -b\\
 -b&-b&-b& -b&1-b
\end{array} \right)
.\]
\end{example}
By using Cramer's method we obtain that the generating function $G(x,q|j)$ is given by
$$\frac{\det A_i}{\det A}$$
Where, $A_i$ is obtained by replacing column number $i$ in matrix $A$ with $\left(\begin{array}{c}
 a\\
 a\\
 \vdots \\
 a \end{array}\right)$.  
 \begin{lemma}\label{L3}
For all $k$ odd 
$$\det A=\sum_{i=0}^{k}{(-1)^{\lfloor\frac{i+1}{2}\rfloor}}\binom{\lfloor\frac{k-i}{2}\rfloor+i}{i}b^i$$
\end{lemma}
\begin{proof}
For $k=1$
\[
\det A= 1
.\]
For $k=3$
\[
\det A= \left|\begin{array}{lll}
 1 & 0&-b\\
 0 & 1-b&-b\\
-b&-b&1-b
\end{array} \right|=b^3-b^2-2b+1
.\]
For $k=5$
\[
\det A= \left| \begin{array}{lllll}
 1 & 0& 0&0&-b\\
 0 & 1&0& -b&-b\\
0 & 0&1-b&-b& -b\\
0 & -b&-b&1-b& -b\\
 -b&-b&-b& -b&1-b
\end{array} \right|=-b^5+b^4+4b^3-3b^2-3b+1
.\]
  By guessing the formula and using induction over $k$ even we get the result.
\end{proof}
\begin{lemma}\label{L4}
For $1\leq\ell\leq k$, where $k$ is odd  
$$\det A_\ell=a\sum_{j=0}^{\ell-1}\left(1+{(-1)^{\lfloor\frac{j+1}{2}\rfloor}}\sum_{i=1}^{j}(-1)^{\lfloor\frac{j-i+1}{2}\rfloor}
\binom{\lfloor\frac{j-i}{2}\rfloor+i}{i}b^i\right)$$
\end{lemma}
\begin{proof}
For $k=1$
\[
\det A_1= a
.\]
For $k=3$
\[
\det A_1= \left|\begin{array}{lll}
 a & 0&-b\\
 a & 1-b&-b\\
a&-b&1-b
\end{array} \right|=a(-b+1)
.\]
\[
\det A_2= \left|\begin{array}{lll}
 1 & a&-b\\
 0 & a&-b\\
-b&a&1-b
\end{array} \right|=a
.\]
\[
\det A_3= \left|\begin{array}{lll}
 1 & 0&a\\
 0 & 1-b&a\\
-b&-b&a
\end{array} \right|=a(-b^2+b+1)
.\]
For $k=5$
\[
\det A_1= \left| \begin{array}{lllll}
 a& 0& 0&0&-b\\
 a& 1&0& -b&-b\\
 a& 0&1-b&-b& -b\\
 a& -b&-b&1-b& -b\\
 a&-b&-b& -b&1-b
\end{array} \right|=a(b^3-b^2-2b+1)
.\]
\[
\det A_2= \left| \begin{array}{lllll}
1 & a& 0&0&-b\\
0 & a&0& -b&-b\\
0 & a&1-b&-b& -b\\
0 & a&-b&1-b& -b\\
-b&a&-b& -b&1-b
\end{array} \right|=a(-b+1)
.\]
\[
\det A_3= \left| \begin{array}{lllll}
 1 & 0& a&0&-b\\
 0 & 1&a& -b&-b\\
0 & 0&a&-b& -b\\
0 & -b&a&1-b& -b\\
 -b&-b&a& -b&1-b
\end{array} \right|=a
.\]
\[
\det A_4= \left| \begin{array}{lllll}
 1 & 0& 0&a&-b\\
 0 & 1&0& a&-b\\
0 & 0&1-b&a& -b\\
0 & -b&-b&a& -b\\
 -b&-b&-b& a&1-b
\end{array} \right|=a(-b^2+b+1)
.\]
\[
\det A_5= \left| \begin{array}{lllll}
 1 & 0& 0&0&a\\
 0 & 1&0& -b&a\\
0 & 0&1-b&-b& a\\
0 & -b&-b&1-b& a\\
 -b&-b&-b& -b&a
\end{array} \right|=a(b^4-b^3-3b^2+2b+1)
.\]

  By guessing the formula and using induction over $k$ even we get the result.
\end{proof}
Using (\ref{eq2}), we derive that the generating function $GC_k(x,q)$ is given by
\begin{align*}
GC_k(x,q)=1+\frac{1}{\det A}\sum_{\ell=1}^{k}{\det A_\ell},
\end{align*}
for $k$ odd.
By repeating the same steps as a Theorem \ref{Th3} and applying lemma \ref{L3} and lemma \ref{L4} we obtain the desired result.

\end{proof}

\begin{corollary}
The number of the gk-connectors blocks over all words of length $n$ over the alphabet $[k]$ 
 where $k$ is odd number,is given by
$$\lfloor \frac{k+1}{2} \rfloor k^{n-1} (n-1).$$
\end{corollary}
\begin{proof}
To find the total number of the gk-connectors blocks over all words of length $n$ over the alphabet $[k]$, we differentiate the generating 
function $GC_k(x,q)$ according to $q$, substitute $q=1$, and extract the coefficients of $x^n$.
Let
 $$f(x,q)=\sum_{i=0}^{k}{(-1)^{\lfloor\frac{i+1}{2}\rfloor}}\binom{\lfloor\frac{k-i}{2}\rfloor+i}{i}x^i(q-1)^i,$$ and $$g(x,q)=f(x,q)-x\sum_{j=0}^{k-1}\left(1+{(-1)^{\lfloor\frac{j+1}{2}\rfloor}}\sum_{i=1}^{j}(-1)^{\lfloor\frac{j-i+1}{2}\rfloor}
\binom{\lfloor\frac{j-i}{2}\rfloor+i}{i}x^i(q-1)^i\right).$$
Note that, according to Theorem \ref{Th4}, we have
 $$GC_k(x,q)=\frac{f(x,q)}{f(x,q)-g(x,q)}.$$ 
Differentiating  $GC_k(x,q)$ according to $q$ and substituting $q=1$, we obtain,
\begin{equation}\label{eq 1}
\frac{\partial GC_k(x,q)}{\partial q}|_{q=1}=\frac{\frac{{\partial f(x,q)}}{{\partial q}}|_{q=1}(f(x,1)-g(x,1))-\frac{\partial( {f(x,q)-g(x,q)})}{\partial q}|_{q=1}f(x,1)}{(f(x,1)-g(x,1))^2}.
\end{equation}
Where,
$f(x,1)-g(x,1)=1-kx$, $\frac{{\partial f(x,q)}}{{\partial q}}|_{q=1}=-\lfloor \frac{k+1}{2}\rfloor x$, and 
$$\frac{\partial {g(x,q)}}{\partial q}|_{q=1}=x^2\sum_{j=1}^{k-1}(-1)^{\lfloor \frac{j+1}{2} \rfloor + \lfloor \frac{j}{2} \rfloor}(\lfloor \frac{j-1}{2} \rfloor +1).$$
Using the facts
$ \lfloor \frac{j+1}{2} \rfloor + \lfloor \frac{j}{2} \rfloor=j$
 and $\lfloor \frac{j-1}{2} \rfloor +1=\lfloor \frac{j+1}{2} \rfloor$
leads to $$\frac{\partial {g(x,q)}}{\partial q}|_{q=1}=x^2\sum_{j=1}^{k-1}(-1)^{j}\lfloor \frac{j+1}{2} \rfloor. $$
Using the identity $\sum_{j=1}^{k-1}(-1)^{j}\lfloor \frac{j+1}{2} \rfloor=0$ 
in (\ref{eq 1}), leads to
\begin{align*}
\frac{\partial GC_k(x,q)}{\partial q}|_{q=1}&= \left( \lfloor \frac{k+1}{2} \rfloor k \right)\frac{x^2}{(1-kx)^2}\\
&=\left( \lfloor \frac{k+1}{2} \rfloor k \right)x^2\sum_{i=0}^{\infty}\binom{1+n}{n}(kx)^n.
\end{align*}
By finding the coefficients of $x^{n-2}$ we get the desired result.
\end{proof}

\end{document}